\documentclass[12pt]{amsart}
\usepackage{amsmath}
\usepackage{amsthm}

\numberwithin{equation}{section}

\newtheorem{theo}{Theorem}[section]
\newtheorem{prop}[theo]{Proposition}
\newtheorem{lemm}[theo]{Lemma}
\newtheorem{cor}[theo]{Corollary}

\theoremstyle{definition}
\newtheorem{defi}[theo]{Definition}
\theoremstyle{remark}
\newtheorem{rem}[theo]{Remark}
\theoremstyle{definition}
\newtheorem{ex}[theo]{Example}
\newcommand{\ddbar}{dd^c}
\newcommand{\OX}{\mathcal{O}_{X}}
\newcommand{\OV}{\mathcal{O}_{V}}
\newcommand{\ac}{\mathrm{ac}}
\newcommand{\reg}{\mathrm{reg}}
\newcommand{\Psh}[1]{ {\rm{Psh}}(#1, \theta | _{#1}) }

\newcommand{\PshX}{ {\rm{Psh}}(X, \theta ) }
\newcommand{\V}[1]{ {\rm{vol}}_{\it{#1}} }

\begin{document}

\title[]
{An ampleness criterion with the extendability\\ of 
singular positive metrics.}
\author{SHIN-ICHI MATSUMURA}

\address{Graduate School of Mathematical Sciences, University of Tokyo, 3-8-1 Komaba,
Tokyo, 153-8914, Japan.}

 \email{{\tt
shinichi@ms.u-tokyo.ac.jp, 
shinichi@sci.kagoshima-u.ac.jp, 
mshinichi0@gmail.com
}}

\thanks{Classification AMS 2010: 32U05, 32C25, 32U40. }


\maketitle

\begin{abstract}
Coman, Guedj and Zeriahi proved that, 
for an ample line bundle $L$ 
on a projective manifold $X$, 
any singular positive metric 
on the line bundle $L|_{V}$ 
along a subvariety $V \subset X$ 
can be extended 
to a global singular positive metric on $L$. 
In this paper, 
we prove that the extendability of singular positive metrics 
on a line bundle along a subvariety implies  
the ampleness of the line bundle.
\end{abstract}

\section{Introduction}
Throughout this paper, let us denote by $X$ a projective manifold of dimension $n$,    
by $L$ a holomorphic line bundle on $X$. 
In the theory of several complex variables and algebraic geometry, 
it is fundamental to consider a singular metric on $L$ whose  
Chern curvature is a positive $(1,1)$-current.
A singular metric on $L$ with positive curvature current 
corresponds to a $\theta $-plurisubharmonic 
function, where 
$\theta $ is a smooth $d$-closed $(1,1)$-form 
which represents the first Chern class $c_{1}(L)$
of the line bundle $L$. 
(For simplicity of notation, we will abbreviate  
\lq \lq $\theta $-plurisubharmonic" 
to \lq \lq $\theta $-psh".) 
Here a function 
$\varphi :X \longrightarrow [-\infty,\infty )$ is called
a {\it{$\theta$-psh function}}, 
if $\varphi$ is upper semi-continuous on $X$ and the Levi form 
$\theta +\ddbar \varphi $ is a positive $(1,1)$-current.
We will denote by $\PshX$ the set of $\theta$-psh functions on $X$.
That is, 
\begin{equation*}
\PshX : = \big\{\varphi :X \longrightarrow [-\infty,\infty )\  
\big| 
\ \varphi \ {\text{ is upper semi-continuous and } }
\theta +\ddbar \varphi \geq 0.\ 
\big\}.
\end{equation*}

It is of interest to know 
when a $\theta|_{V} $-psh function 
on a (closed) subvariety $V \subseteq X$ can be extended to 
a global $\theta $-psh function on $X$.
Coman, Guedj and Zeriahi proved that 
a $\theta |_{V}$-psh function 
on any subvariety $V$ can be extended 
a global $\theta $-psh function on $X$, 
if $L$ is an ample line bundle (see \cite[Theorem B]{CGZ}).
Note that a $\theta| _{V}$-psh function can be defined
even if $V$ has singularities 
(see \cite[Section 2]{CGZ} for the precise definition).

\begin{theo}[{\cite[Theorem B]{CGZ}}] \label{CGZ}
Let $L$ be an ample line bundle over a projective manifold 
$X$ and let $\theta $ be a K\"ahler form representing 
$c_{1}(L)$. Then for any subvariety $V \subseteq X$, 
any $\theta |_{V} $-psh function on $V$ extends to 
a $\theta$-psh function on $X$.  
\end{theo}
The following theorem asserts that 
the converse implication of Theorem \ref{CGZ} holds, 
which is a main result in this paper. 
Theorem \ref{Main} gives an ampleness criterion by the 
extendability of singular metrics ($\theta $-psh functions).

\begin{theo}\label{Main}
Let $L$ be a pseudo-effective line bundle whose 
first Chern class $c_{1}(L)$ is not zero. 
Assume that $L$ satisfies the following property\ $:$ 
For any subvariety $V$ and 
any $\theta |_{V}$-psh function
$\varphi \in \Psh{V}$, 
there exists a global $\theta $-psh function 
$\widetilde{\varphi} \in \PshX$ such that
$\widetilde{\varphi}|_{V} = \varphi$.
Here $\theta $ is a smooth $d$-closed $(1,1)$-form 
representing $c_{1}(L)$.
Then $L$ is an ample line bundle.
\end{theo}

A line bundle $L$ is called {\it{pseudo-effective}} if 
$\PshX$ is not empty. 
We can easily check that the definition of pseudo-effective line bundles 
does not depend on the choice of $\theta \in c_{1}(L)$. 
In the proof of Theorem \ref{Main}, we 
consider only the case 
when $V$ is a strongly movable curve (see section 3).
Here a curve $C$ is called a strongly movable curve if 
\begin{equation*}
C=\mu _{*} \big( A_{1} \cap \dots \cap A_{n-1} \big)
\end{equation*}
for suitable very ample divisors $A_{i}$ on $\tilde X$, 
where $\mu :\tilde X \to X$ is a birational morphism. 
See \cite[Definition 1.3]{BDPP} for more details. 
Thus for an ampleness criterion, 
it is sufficient to check the extendability 
from a strongly movable curve.

It is important to emphasize that 
even if a given $\theta |_{V}$-psh function
$\varphi$ is smooth at some point on $V$, 
the extended function $\widetilde{\varphi}$ 
may not be smooth at the point.
The fact seems to make the proof of Theorem \ref{Main} 
difficult. 

Remark that the assumption that 
the first Chern class $c_{1}(L)$ 
is not zero is necessary.
Indeed,  
when the first Chern class $c_{1}(L)$ is zero and 
$\theta $ is equal to zero as a $(1,1)$-form, 
a $\theta $-psh function is always constant, 
from the maximum principle of psh functions.
Hence any $\theta |_{V}$-psh function can be extended.
However $L$ is not an ample line bundle.
In other words, a line bundle which satisfies 
the extendability condition in Theorem \ref{Main} 
must be ample or numerically trivial 
(that is, $c_{1}(L)$ is zero).

This paper is organized in the following way: 
In section 2, we collect materials 
to prove Theorem \ref{Main}.
Section 3 is devoted to give the proof of Theorem \ref{Main}.
In section 4, there are two examples which 
give ideas for the proof of Theorem \ref{Main}.

\subsection*{Acknowledgment}
The author would like to thank 
his supervisor Professor Shigeharu Takayama 
for carefully reading a preliminary version of this article and 
useful comments.
He is indebted to Dr. 
Tomoyuki Hisamoto who suggests him to consider  
the converse implication of Theorem \ref{CGZ}.
He wishes to thank Professor Ahmed Zeriahi
for warm encouragement and comments.
He obtained an opportunity to discuss the problem 
with Professor Ahmed Zeriahi 
when he attended the conference 
\lq \lq Complex and Riemannian Geometry" at CIRM. 
He is grateful to the organizers.
He is supported by the Grant-in-Aid for Scientific Research 
(KAKENHI No. 23-7228)
and the Grant-in-Aid for JSPS fellows.

\section{Preliminaries}
In this section, we collect materials for the proof of 
Theorem \ref{Main}.
The propositions in this section may be known facts. 
However we give comments or references 
for the readers' convenience.

\begin{lemm}\label{inter}
Let $L$ be a pseudo-effective line bundle 
whose first Chern class $c_{1}(L)$ is not zero.
Then the intersection number $(L \cdot B^{n-1})$ is positive
for any ample line bundle $B$ on $X$.
\end{lemm}
\begin{proof}
Let $\theta $ be a smooth $d$-closed $(1,1)$-form 
which represents the first Chern class $c_{1}(L)$. 
We take an arbitrary smooth $(n-1,n-1)$-form $\gamma $ on $X$. 
Further we take a K\"ahler form $\omega $ which represents 
the first Chern class $c_{1}(B)$.
Since the $(n-1,n-1)$-form $\omega ^{n-1}$ is 
strictly positive, 
there exists a large constant $C>0$ such that 
$$-C  \omega ^{n-1} \leq \gamma \leq C  \omega ^{n-1}.$$
Here we implicitly used the compactness of $X$.
Since $L$ is pseudo-effective, 
we can take a function $\varphi $ in $\PshX$.
The Levi form $\theta + \ddbar{\varphi }$ is 
a positive $(1,1)$-current. 
It gives the following inequality:
\begin{align*}
-C \int_{X}(\theta + \ddbar{\varphi }) \wedge \omega ^{n-1}
\leq \int_{X} (\theta + \ddbar{\varphi }) \wedge \gamma 
\leq C \int_{X} (\theta + \ddbar{\varphi })  \wedge \omega ^{n-1}.
\end{align*} 
The $(1,1)$-form $\theta $ (reps. $\omega$) represents 
the first Chern class of $L$ (resp. $B$). 
Thus the integral of the right and left hands 
\begin{equation*}
\int_{X} (\theta + \ddbar{\varphi }) \wedge \omega ^{n-1}
\end{equation*}
agrees with the intersection number $( L \cdot B^{n-1} )$. 
If the intersection number is zero, 
\begin{equation*}
\int_{X} (\theta + \ddbar{\varphi } ) \wedge \gamma = 0
\end{equation*} 
for any smooth $(n-1,n-1)$-form $\gamma $ from the above inequality.
It means that the $(1,1)$-current 
$ \theta + \ddbar{\varphi } $ is a zero current.
This is a contradiction to the assumption that 
$c_{1}(L)$ is not zero.
Hence the intersection number $( L \cdot B^{n-1} )$ is 
positive for any ample line bundle $B$.
\end{proof}

\begin{theo}[{\cite[Theorem 1.3]{Zh09}}]
\label{curve}
Let $X$ be a projective manifold
of dimension at least two and $p$, $q$ be points on $X$.
Fix an ample line bundle $B$ on $X$.
Then there exists a smooth curve $C$
with the following properties{\rm:}\\
\ \ $(1)$ $C$ is a complete intersection of 
the complete linear system $|mB|$ for some $m>0$. \\
\ \ $(2)$ $C$ contains points $p$ and $q$.
\end{theo}
\begin{proof}
Take an embedding of $X$ into the projective space 
$\mathbb{P}^{N}$.
Then two points $p$, $q$ are always in general 
position in $\mathbb{P}^{N}$ (see [Zh09, Definition 1.1] for the definition).
$\rm{[Zh09, Theorem 1.3]}$ asserts that 
a general member of $|mB|_{p,q}$ is irreducible and smooth, where 
$|mB|_{p,q}$ is a linear system
in $|mB|$ passing through $p$ and $q$.
Then by taking a complete intersection of 
general members of $|mB|_{p,q}$, 
we can construct a curve with the above properties.
\end{proof}

\begin{lemm}\label{potential}
Let $C$ be an irreducible curve on $X$ and $p$ be 
a non-singular point on $C$.
Assume that the intersection number $(L\cdot C)$ is positive 
$($that is, the restriction $L|_{C}$ is ample$)$.
Then there exists a function $\varphi$ on $C$ with 
the following properties{\rm:}\\
\ \ $(1)$ $\varphi \in \Psh{C}$\\
\ \ $(2)$ The function $\varphi$ has pole at 
$p$ $($that is, $\varphi(p)= -\infty$$)$
and is smooth 
except $p$ .
\end{lemm}
\begin{proof}
By the assumption, $L|_{C}$ is ample on $C$.
Therefore we can obtain a smooth 
strictly $\theta |_{C}$-psh function $\varphi_{1}$ on $C$. 
Even if $C$ has singularities, we can obtain such function.
In fact, there exists an integer $m_{0} > 0$ such that 
the complete linear system of $m_{0}L|_{C}$
gives an embedding of $C$ to the projective space 
$\mathbb{P}^{N}$, since $L|_{C}$ is ample. 
Now there exists a smooth strictly $\theta_{0}$-psh function 
$\psi$ on $\mathbb{P}^{N}$. 
Here $\theta_{0}$ is a $(1,1)$-form representing the 
first Chern class of the hyperplane bundle 
$\mathcal{O}_{\mathbb{P}^{N}}(1)$ on $\mathbb{P}^{N}$. 
Since the restriction to $C$ 
of $\mathcal{O}_{\mathbb{P}^{N}}(1)$ is equal to 
$m_{0}L|_{C}$, the function $ (1/m_{0}) \psi |_{C}$
gives a smooth strictly $\theta |_{C}$-psh function on $C$.

Let $z$ be a local coordinate on $C$ centered at $p$.
We define a function $\varphi_{2}$ on $C$ to be 
$ \varphi_{2}: = \rho \log |z|^{2} $, where 
$\rho$ is a smooth function on $C$ 
whose support is contained in some neighborhood of $p$.
Then $ \varphi_{2}$ has a pole only at $p$. 
Further $ \varphi_{2}$ is an almost psh function 
(that is, there exists a smooth $(1,1)$-form $\gamma$ such that 
$\ddbar  \varphi_{2} \geq \gamma$).
Then a function $\varphi$ which is defined to be 
$\varphi = (1-\varepsilon)\varphi_{1} + \varepsilon \varphi_{2}$ 
satisfies the above condition 
for a sufficiently small $\varepsilon>0$. 
In fact, the property (1) follows from the strictly positivity of 
the Levi-form of $\varphi_{1}$.
The function $\varphi$ has a pole only at $p$ thanks to $\varphi_{2}$. 
\end{proof}

Lemma \ref{inter}, \ref{potential} say that 
there exists many $\theta |_{C}$-psh functions 
on a complete intersection of very ample divisors.

\begin{lemm}\label{max}
Let $\gamma $ be a smooth $d$-closed $(1,1)$-form on $X$ 
and a function $\varphi_{i}$ $($for $i=1,2)$ 
be a $\gamma$-psh function on $X$.  
Then the function $\max(\varphi_{1}, \varphi_{2})$ 
is also a $\gamma$-psh function on $X$.
\end{lemm}
\begin{proof}
First we remark $\gamma$-plurisubharmonicity is a local property.
We can locally take a smooth potential function 
of $\gamma $ since $\gamma $ is a $d$-closed $(1,1)$-form.
Thus we can locally write $\gamma = \ddbar \psi$ 
for some function $\psi $.
By the assumption, $\ddbar( \psi + \varphi_{i} )$ is a positive 
current.
Therefore the Levi form of  
\begin{equation*}
\max(\psi + \varphi_{1}, \psi + \varphi_{2})= 
\psi + \max(\varphi_{1}, \varphi_{2})
\end{equation*}
is also a positive current.
It means that $\gamma + \ddbar 
\max(\varphi_{1}, \varphi_{2} )\geq 0$.
Upper semi-continuity of functions is preserved. 
Hence the function $\max(\varphi_{1}, \varphi_{2} )$ is 
a $\gamma$-psh function.
\end{proof}

For the proof of Theorem \ref{Main}, 
it is important to obtain strict positivity from
extended $\theta$-psh functions.
The main idea for the purpose is to 
use the volume of a line bundle and 
its expression formula in terms of current integration, 
which is proved in \cite{Bou02}.
\begin{defi}
Let $M$ be a line bundle on 
a projective variety $Y$ of dimension $d$.
Then the volume of $M$ on $Y$ is defined to be 
\begin{equation*}
\V{Y}(M) = \limsup_{k \to \infty} \dfrac
{ \dim H^{0}\big( Y,\mathcal{O}_{Y}(kM) \big) }{k^d/d!}.
\end{equation*}
\end{defi}
The volume asymptotically measures  
the number of global holomorphic sections. 
The volume of a line bundle can be defined for 
a $\mathbb{Q}$-line bundle and, 
depends only on 
the numerical class (the first Chern class) of the line bundle. 
Moreover the volume is a continuous function on 
the $\mathbb{Q}$-vector space $N^{1}(Y)_{\mathbb{Q}}$ of 
the numerical equivalent classes of $\mathbb{Q}$-line bundles. 
(See \cite[Proposition 2.2.35, 2.2.41]{La04} for the precise statement.) 
The above properties are used in the proof of Theorem \ref{Main}.

The following proposition gives a relation 
between the volume and curvature currents of 
a line bundle.
It is proved by using singular Morse inequalities (which is proved 
in \cite{Bon98}) and approximations of $\theta$-psh functions 
(see \cite{Bou02} for details).

\begin{prop}[{\cite[Proposition 3.1]{Bou02}}]
\label{Bou}
Let $M$ be a pseudo-effective line bundle on 
a projective manifold $Y$ of dimension $d$ and 
$\eta $ a smooth $(1,1)$-form which represents the 
first Chern class $c_{1}(M)$ of $M$.
Then for any $\eta $-psh function $\varphi $ on $Y$, 
we have
$$ \V{Y}(M) \geq \int_{Y} 
\big(\eta + \ddbar \varphi \big)_{\ac}^{d}.
$$
\end{prop}
Here $\big(\eta + \ddbar \varphi \big)_{\ac}$ means the absolutely continuous part
of a positive current $\big(\eta + \ddbar \varphi \big)$ by 
the Lebesgue decomposition.  
(See \cite[Section2]{Bou02} for the precise definition.)
We use only property that 
when $\varphi$ is smooth on an open set, the equality 
\begin{equation*}
\big(\eta + \ddbar \varphi \big)_{\ac} = \big(\eta + \ddbar \varphi \big)
\end{equation*}
holds on the open set. 

Actually, the above inequality would be an equality 
by taking supremum of the right-hand side over 
all $\eta $-psh functions (see \cite[Theorem 1.2]{Bou02}).
It is generalized to the restricted volume along a subvariety
(cf. \cite[Theorem 1.2]{Mat10}).
These expressions of the volume and restricted volume 
with current integrations give an example, which show 
us that there exists a $\theta $-psh function on some subvariety
which can not be extended to a global $\theta $-psh function 
even if $L$ is a big line bundle 
(see Example \ref{Ex2}).

In section 3, we need to approximate
a given $\theta $-psh function
by almost psh functions with mild singularities.
For the purpose, we use Theorem \ref{appro}. 
Theorem \ref{appro} says that it is possible to approximate a given almost psh function
with the same singularities as a logarithm of a sum of squares of holomorphic functions 
without a large loss of positivity of the Levi form.

\begin{theo}[{\cite[Theorem 13.12]{Dem}}]
\label{appro}
Let $\varphi $ be an almost psh function on a compact complex manifold $X$ such that $ \ddbar \varphi >\gamma $ for some continuous $(1,1)$-form $\gamma $. 
Fix a hermitian form $\omega$ on $X$.
Then there exists a sequence of
almost psh functions $\varphi _{k}$ and a decreasing sequence 
$\delta_{k} > 0$ converging to $0$
with the following properties$:$\\
{\rm (A)}\ \ $\varphi (x) < \varphi _{k}(x) \leq 
\sup_{|\zeta -x|<r}\varphi (\zeta ) + 
C \Big( \dfrac{| \log r|}{k}+ r + \delta_{k} \Big)$\\
\ \ \ \ \ \ \ with respect to coordinate open 
sets covering $X$.\\
{\rm (B)}\ \ $\varphi _{k}$ has the same singularities 
as a logarithm of a sum of squares of 
holomorphic functions.
In particular, $\varphi _{k}$ is 
smooth except the polar set of $\varphi$.\\
{\rm (C)}\ \ 
$\ddbar \varphi _{k}\geq  \gamma -\delta_{k} \omega.$

\end{theo}

\section{Proof of Theorem \ref{Main}}
In this section, we give the proof of Theorem \ref{Main}.
Let $L$ be a line bundle 
with the assumption of Theorem \ref{Main} and 
$\theta$ be a smooth $d$-closed $(1,1)$-form which 
represents the first Chern class $c_{1}(L)$. 
According to the Nakai-Moishezon-Kleiman criterion 
(cf. \cite[Theorem 1.2.23]{La04}), 
in order to show that $L$ is ample, 
it is enough to see 
the self-intersection number $(L^{d} \cdot V)$ along $V$
is positive for any irreducible subvariety $V$. 

For the purpose, we first show the following proposition 
which implies that 
the self-intersection number along 
an irreducible curve is always positive. 

\begin{prop}\label{Nef}
Let $L$ be a line bundle with 
the assumption in Theorem \ref{Main} and
$V$ be an (irreducible) subvariety on $X$.
Then \\
\ \ $(1)$ The restriction $L|_{V}$ to $V$ is pseudo-effective.\\
\ \ $(2)$ The restriction $L|_{V}$ to $V$ is not numerically trivial.
\end{prop}

\begin{rem}
When $V$ is non-singular, the property (2) means 
that the first Chern class $c_{1}(L|_{V})$ is not zero.
\end{rem}
\begin{proof}
First we 
take two different points $p$, $q$ on $V_{\text{reg}}$. 
Here $V_{\text{reg}}$ means the regular locus of $V$.
Then we can take a smooth curve $C$ on $X$ 
such that $C$ contains $p$, $q$
by Theorem \ref{curve}.
By the construction, the curve $C$ is a complete intersection of 
the linear system of some very ample line bundles.
It follows that the intersection number  
$(L \cdot C)$ along $C$ is positive from Lemma \ref{inter}.
Lemma \ref{potential} asserts that 
there exists a function $\varphi \in \Psh{C}$ 
such that $\varphi$ has a pole at $p$ and is smooth at $q$. 
The $\theta |_C{}$-function $\varphi$ on $C$ can be extended 
to a global $\theta$-function 
on $X$ by the assumption of the extendability. 
The extended function to $X$ does not have a pole at $q \in V$.
It means that the restriction to $V$ of the function is 
well-defined 
(that is, the function is not identically $-\infty$ on $V$).
We denote by $\widetilde {\varphi}$ 
the restriction  to $V$ of the function.
The function $\widetilde {\varphi}$ gives an element in $\Psh{V}$.
Hence $L|_{V}$ is pseudo-effective.

From now on, 
we show that $L|_{V}$ is not numerically trivial.
For a contradiction, we assume that $L|_{V}$ is numerically trivial.
First we consider the case when $V$ is non-singular.
Then there exists a function on $V$ such that 
\begin{equation*}
\theta |_{V}+ \ddbar \widetilde {\varphi} = \ddbar \psi. 
\end{equation*}
from the $\partial \bar{\partial}$-Lemma, 
since $L|_{V}$ is numerically trivial 
(that is, first Chern class $c_{1}(L|_{V})$ is zero).
Since the function $\widetilde {\varphi}$ is 
a $\theta |_{V}$-psh, $\psi$ is a psh function on $V$. 
It follows that $\psi$ is actually a constant 
by the maximum principle of psh functions. 
Therefore $\theta |_{V}+ \ddbar \widetilde {\varphi}$ 
is a zero current.
We know that $\theta $-pluriharmonic functions are always smooth.
Hence the function $\widetilde {\varphi}$ is smooth on $V$.
However $\widetilde {\varphi}$ has a pole at $p$
by the construction.
This is a contradiction.

We need to consider the case when $V$ has singularities. 
Then we take an embedded resolution
\begin{equation*}
\mu : \widetilde {V} \subseteq  \widetilde {X}  \longrightarrow V \subseteq X
\end{equation*}
of $V \subseteq X$.
That is, $\mu :\widetilde {X} \longrightarrow X$ is a 
birational morphism and the restriction of $\mu$ 
to $\widetilde {V}$ gives a resolution of 
singularities of $V$. 
Since $p$ is contained in the regular locus of $V$, 
$\mu$ is an isomorphism on some 
neighborhood of $p$.
Further the pull-back $(\mu^{*}L)|_{\widetilde{V}} $ is also
numerically trivial since $L|_{V}$ is numerically trivial. 
The same argument asserts that 
any function in 
$\rm{Psh}(\widetilde{V}, (\mu^{*}\theta) | _{\widetilde{V}})$ is always smooth. 
Note that the pull-back $\mu^{*} \widetilde {\varphi}$ is 
a $(\mu^{*} \theta) | _{\widetilde{V}}$-psh function on $V$. 
It shows that the pull-back $\mu^{*} \widetilde {\varphi}$ is 
smooth on $\widetilde{V}$. 
The function $\widetilde{\varphi }$ is also smooth at $p$
since $\mu$ is an isomorphism on some 
neighborhood of $p$.
However $\widetilde {\varphi}$ has a pole at $p$ by the construction.
This is a contradiction. 
Thus $L|_{V}$ is not numerically trivial 
even if $V$ has singularities.
\end{proof}

\begin{cor} Let $L$ be a line bundle with assumption in Theorem \ref{Main}. 
Then the intersection number $(L\cdot C)$ is positive 
for any irreducible curve $C$ on $X$.
\end{cor}
\begin{proof}
Any pseudo-effective line bundle on a curve 
which is not numerically trivial is always ample. 
Thus, the corollary follows from Proposition \ref{Nef}.
\end{proof}
In oder to show $(L^{d} \cdot V)>0$ for any subvariety, 
we need only consider the case when the dimension of $V$ is larger than 
or equal to two from the above corollary.
Moreover the above corollary asserts that $L$ is 
a nef line bundle on $X$.
It is well-known that the volume of $L|_{V}$ is equal to the self-intersection 
number $(L^{d} \cdot V)$ along $V$ for a nef line bundle.
(see \cite[Section 2.2 C]{La04}).  
That is, the equality holds 
$$\V{V} (L) =(L^{d} \cdot V)$$ 
for any irreducible subvariety $V$ of dimension $d$. 
(Note the restriction of a nef line bundle is also nef.) 
Therefore for the proof of Theorem \ref{Main}, 
it is enough to show that the volume 
$\V{V}(L)$ is always positive 
for any irreducible subvariety $V$ of dimension $d\geq 2$.
From now on, we will show that the volume $\V{V}(L)$ is positive 
for a subvariety $V$ of dimension $d\geq 2$, 
by using Proposition \ref{Bou}.  

We first consider the case when $V$ is non-singular.  
Even if $V$ has singularities, the same argument can be justified 
by taking an embedded resolution of $V \subseteq X$. 
We argue the case at the end of this section.

Fix a point $p$ on $V$.
Let $(z_{1}, z_{2}, \dots, z_{d})$ be a local coordinate 
centered at $p$.
We consider an open ball $B$, which is defined by 
\begin{equation*}
B:=\big\{(z_{1}, z_{2}, \dots, z_{d}) 
\ \big|\ |z|^{2} < 1 \big\}. 
\end{equation*}
Since $\ddbar |z|^{2}$ is a strictly positive $(1,1)$-form on $B$, 
there exists a large positive number $A$ such that 
\begin{equation}
A \ddbar |z|^{2} + \theta |_{V} >0\ \ \ \ \text{on}\ B.
\label{0-0}
\end{equation}

For every point $y$ on the boundary $\partial B$ of $B$, 
we can take a curve $C_{y}$ on $V$ such that 
$C_{y}$ contains $p$ and $y$.
We can take such curve from 
Theorem \ref{curve} and the assumption that 
the dimension of $V$ is larger than or equal to two.
By Lemma \ref{inter} and the property (1) in Theorem \ref{curve}, the restriction of $L$ to $C_{y}$ is ample.
Therefore there exists a function 
$\varphi _{y} $ on $C_{y}$ with following properties:
\begin{align}
&\varphi _{y} \in \Psh{C_{y}}, \\
&\varphi _{y}(p) = -\infty,  \\
&\varphi _{y}(y) = 0, \label{0-1}  
\end{align}
Indeed, we can take a function $\varphi _{y} \in \Psh{C_{y}}$
such that $\varphi _{y}$ has a pole at $p$ and 
$\varphi _{y}$ is smooth at $y$ by Lemma \ref{potential}.
After replacing $\varphi _{y}$ by $\varphi _{y}- \varphi _{y}(y)$, 
the function satisfies the property (\ref{0-1}).
Note that the function is a $\theta |_{C_{y}}$-psh function
even if we replace $\varphi _{y}$ by $\varphi _{y}-\varphi _{y}(y)$.

Now the function ${\varphi}_{y}$ on $C_{y}$ 
can be extended to 
a $\theta |_{V}$-psh function $\widetilde{\varphi}_{y}$
 on $V$ by the assumption of 
Theorem \ref{Main}. 
In fact, we can extend ${\varphi}_{y}$ to 
a global $\theta$-psh function on $X$ by the assumption 
in Theorem \ref{Main}.
From the property (\ref{0-1}), 
the extended function does not have pole at $y$.
Thus we can restrict the function to $V$. 
The function gives the extension of ${\varphi}_{y} $ to $V$, 
which we denote by $\widetilde{\varphi}_{y}$.
Then the function $\widetilde{\varphi}_{y}$ 
satisfies the following properties:
\begin{align}
&\widetilde{\varphi}_{y} \in \Psh{V}, \label{1-1} \\
&\widetilde{\varphi}_{y}(p) = -\infty, \label{1-2} \\
&\widetilde{\varphi}_{y}(y) = 0. \label{1-3}
\end{align}

In the following step, we approximate the function 
$\widetilde{\varphi}_{y}$ with the same singularities as 
a logarithm of a sum of squares of holomorphic functions.
If the extended function $\widetilde{\varphi}_{y}$ is continuous 
on some neighborhood of $y$, this step is not necessary.
However the function $\widetilde{\varphi}_{y}$ 
may not be continuous at $y$ 
even if $\varphi_{y}$ is smooth at $y$ on $C_{y}$.
Thus the following step seems to be necessary in general.

\begin{lemm}\label{Key}
Fix a hermitian form $\omega$ on $X$. 
For every positive number $\varepsilon $ and a point 
$y \in \partial B$, 
there exist a neighborhood $U_{y}$ of $p$
which is independent of $\varepsilon $
and 
an almost psh $\widetilde{\varphi}_{y, \varepsilon }$
with following properties$:$
\begin{align}
&\theta |_{V} + \ddbar \widetilde{\varphi}_{y, \varepsilon }
\geq -\varepsilon \omega,\label{2-1} \\
&\widetilde{\varphi}_{y, \varepsilon }(y) > 0 \label{2-2} 
\ \text{and $\widetilde{\varphi}_{y, \varepsilon }$ is 
smooth on some neighborhood of $y$. }\\
&  -A>\widetilde{\varphi}_{y, \varepsilon }
\ \ \ \ {\rm{on}}\ U_{y}. \label{2-3}
\end{align}
\end{lemm}
\begin{proof}
By applying Theorem \ref{appro} to 
$\varphi = \widetilde{\varphi}_{y}$ and 
$\gamma = - \theta |_{V}$, 
we obtain almost psh functions 
$\{ \widetilde{\varphi}_{y, k } \}_{k=1}^{\infty}$
with the properties in Theorem \ref{appro}.
For a given $\varepsilon $, 
by taking a sufficiently large $k=k(\varepsilon ,y)$, 
the property (\ref{2-1}) holds from the property (C).

From the left side inequality of the property (A) 
in Theorem \ref{appro} and 
(\ref{1-3}), 
we can easily check the property (\ref{2-2}) for every 
positive integer $k$.
In fact the property $(B)$ implies that 
if $\widetilde{\varphi}_{y, k }$ does not have a pole 
at a point, 
$\widetilde{\varphi}_{y, k }$ is smooth at the point .
In particular, $\widetilde{\varphi}_{y, k }$ is smooth 
on some neighborhood of $y$.
In order to show the existence of $U_{y}$ with 
the property (\ref{2-3}), we estimate 
the right hand inequality of the property (A).
We can easily show that 
there exists a sufficiently small $r_{1} > 0$ which does not 
depend on $\varepsilon$
such that 
\begin{equation}
0 < C(\dfrac{|\log {r_{1}}|}{k} + r_{1} +{\delta  _{k}})< A 
\ \ \ \ \text{for any }k \geq \lceil \dfrac{1}{r_{1}} \rceil. \label{5}
\end{equation}
Here $\lceil \cdot  \rceil$ means round up of a real number.
Indeed, for any $k \geq \lceil \dfrac{1}{r_{1}} \rceil$,  we have 
\begin{equation*}
\dfrac{|\log {r_{1}}|}{k} + r_{1} +{\delta  _{k}}
\leq r_{1}|\log {r_{1}}| + r_{1} +{\delta  _{k}}.
\end{equation*}
Now $C$ depends on the choice of coordinate open sets 
covering $V$. 
However $C$ is independent of $\varepsilon$. 
(We may assume that 
the coordinate open set 
$\big( B, (z_{1}, z_{2}, \dots, z_{d}) \big)$ is 
a member 
of coordinate open sets covering $V$.)
Therefore the inequality (\ref{5}) holds 
for a sufficiently small $r_{1}$ 
which is independent of $\varepsilon$. 

On the other hand, 
$\widetilde{\varphi}_{y}$ has a pole at $p$ 
by (\ref{1-2}). 
Thus we have
\begin{equation*}
\sup_{|z -z(p)| < r_{2}}
\widetilde{\varphi}_{y}(z) <-2A 
\end{equation*}
for a sufficiently small $r_{2}>0$.
Here we used upper semi-continuity of  
$\widetilde{\varphi}_{y}$.
Then we define $U_{y}$ to be 
$$U_{y}:=
\big\{z \in B \ \ \big| \ |z-z(p)|<r_{3} \big\},$$ 
where $r_{3}$ is $\min \{r_{1}, r_{2}\}$.
Then the right hand of the property (A) in Theorem \ref{appro}
is strictly smaller than $A$ for any $k \geq \lceil \dfrac{1}{r_{3}} \rceil$. 
We emphasize that $r_{1}$ and $r_{2}$ do not depend on $\varepsilon$.
Therefore we obtain $U$ with the property (\ref{2-3}). 
\end{proof}
By using these functions, 
we construct an almost psh function whose 
value at $p$ is smaller than values on the boundary $\partial B$ of $B$.
From the property (\ref{2-2}), 
there exists 
a neighborhood $W_{y}$ of $y$ such that 
\begin{equation*}
\widetilde{\varphi}_{y, \varepsilon }>0
\ \ \ {\text {on}}\ W_{y}.
\end{equation*}
Since $\partial B$ is a compact set, 
we can cover $\partial B$ by 
finite members $\big\{ W_{y_{i}} \big\}_{i=1}^{N}$.
Now we define a function $\Phi _{\varepsilon }$ to be 
$$\Phi _{\varepsilon } := 
\max_{i=1, \dots. N}\big\{ 
\widetilde{\varphi}_{y_{i}, \varepsilon }  \big\}.$$
\begin{lemm}\label{Key2}
Then the function $\Phi _{\varepsilon }$ satisfies 
the following properties:
\begin{align}
& \theta |_{V} + \ddbar \Phi _{\varepsilon } 
\geq - \varepsilon \omega .\label{3-1} \\ 
& \Phi _{\varepsilon } > 0\ 
\textrm{on some neighborhood}\ V_{\varepsilon }\ \textrm{of}\ \partial B.
\label{3-2} \\
& -A > \Phi _{\varepsilon } \ 
\textrm{on some neighborhood}\ U\ \textrm{of}\ p.
\label{3-3}
\end{align}
\end{lemm}
\begin{proof}
The property (\ref{3-1}) follows from Lemma \ref{max} and 
the property (\ref{2-1}).
The property (\ref{3-2}) is clear 
by the definition of $\Phi _{\varepsilon }$ and 
the property (\ref{2-2}).
If a neighborhood $U$ is defined 
to be $U:= \cap_{i=1}^{N} U_{y_{i}} $, 
the property (\ref{3-3}) holds from the property (\ref{2-3}). 
Here $U_{y_{i}}$ is a neighborhood of $p$ 
with the  property (\ref{2-3}) in Lemma \ref{Key}.
\end{proof}
\begin{rem}
We can assume that 
a neighborhood $U$ in the property (\ref{3-3}) does 
not depend on $\varepsilon $.
It follows from the definition of $U$ in the proof of 
Lemma \ref{Key2}.
The fact is essentially important in the estimation of the volume 
$\V{V}(L)$ with current integrations.
\end{rem}
We want to construct a almost psh function 
whose Levi-form is strictly positive on some neighborhood
of $p$.
The integral of the Levi-form would imply that the volume 
$\V{V}(L)$ is positive.
For this purpose, 
we define a new function $\Psi _{\varepsilon }$ on $V$ 
as follows:
\begin{eqnarray}
\Psi _{\varepsilon }:=\left\{ \begin{array}{ll}
\Phi _{\varepsilon }\ \ \ \ \ \ \ \ &\text{on}\ V\setminus B \\
\max \big\{ \Phi _{\varepsilon }, A|z|^{2}-A \big\}
\ \ \ \ \ \ \ &\text{\ on }\ B. \label{6}
\end{array} \right.
\end{eqnarray}
Then the function $\Psi   _{\varepsilon }$ satisfy 
the following properties:

\begin{lemm}\label{Key3}
The function $\Psi   _{\varepsilon }$ satisfies 
the following properties:
\begin{align}
&\theta |_{V} + \ddbar \Psi   _{\varepsilon } 
\geq - \varepsilon \omega , \label{4-1} \\ 
&\Psi _{\varepsilon } = A |z|^{2}-A\ 
\text{on}\ U,\label{4-2}  
\end{align}
\hspace{0.1cm} where $U$ is a neighborhood of $p$ 
which is independent of $\varepsilon$.
\end{lemm}
\begin{proof}
From the property (\ref{3-3}), we have 
$\Phi _{\varepsilon } < -A$ for some neighborhood $U$ of $p$ 
which is independent of $\varepsilon$.
Therefore the inequality 
$$\Phi _{\varepsilon } < -A \leq A |z|^{2}-A$$ 
holds on $U$. 
Thus the property (\ref{4-2}) holds.

Further by the choice of $A$, the $(1,1)$-form 
$A \ddbar |z|^{2} + \theta |_{V} $ is strictly positive 
on the neighborhood $B$ of $p$.
In particular, 
\begin{equation*}
A \ddbar |z|^{2} + \theta |_{V} \geq -\varepsilon \omega
\end{equation*} holds.
Hence 
we have $$\theta |_{V} + \ddbar 
\max \big\{ \Phi _{\varepsilon }, A|z|^{2}-A \big\}
\geq -\varepsilon \omega \ \ \ \  \text{on}\ B$$
from Lemma \ref{max} and the property (\ref{3-1}).

On the other hand, we obtain 
$$\max \big\{ \Phi _{\varepsilon }, A|z|^{2}-A \big\} = 
\Phi _{\varepsilon }$$
on some neighborhood of $\partial B$
from the property (\ref{3-2}).
Therefore the function 
$\Psi _{\varepsilon }$ satisfies
\begin{equation*}
\theta |_{V} + \ddbar \Psi   _{\varepsilon } = 
\theta |_{V} + \ddbar \Phi   _{\varepsilon } \geq 
-\varepsilon \omega
\end{equation*}
on the neighborhood of $\partial B$ from the property (\ref{3-1}).
Thus the property (\ref{4-1}) holds on $X$.
\end{proof}
Finally, we estimate the volume $\V{V}(L)$ of $L$ 
with current integrations for the computation of the intersection number $(L^{d}\cdot V)$.
The function $\Psi _{\varepsilon }$ is a 
$(\theta |_{V} + \varepsilon \omega )$-psh function
by the property (\ref{4-1}).
Here $\omega$ can be assumed to be 
a K${\rm{\ddot{a}}}$hler form which 
represents the first Chern class $c_{1}(B)$ of $B$, 
where $B$ is an ample line bundle on $V$. 
The $d$-closed $(1,1)$-form 
$(\theta |_{V} + \varepsilon \omega)$ represents the first Chern
class $c_{1}(L)+\varepsilon c_{1}(B)$.
Thus by Proposition \ref{Bou}, 
we have 
\begin{equation*}
\V{{V}}
\big( L+ \varepsilon B  \big) 
\geq 
\int_{V} 
\big( \theta |_{V} +  \varepsilon \omega +
\ddbar \Psi  _{\varepsilon } \big)_{\ac}^{d}.
\end{equation*}
Since $(\theta |_{V} +  \varepsilon \omega +
\ddbar \Psi  _{\varepsilon } \big)$ is a positive current, 
the absolute continuous part is (semi)-positive.
It shows 
\begin{equation*}
\V{V}
\big( L + \varepsilon B \big) \geq 
\int_{U} 
\big(\theta |_{V}+  \varepsilon \omega +
\ddbar \Psi _{\varepsilon } \big)_{\ac}^{d}, 
\end{equation*}
where $U$ is a neighborhood of $p$ which satisfies 
the properties in Lemma \ref{Key3}.
If we let $\varepsilon $ tend to zero,
the left hand of the above inequality converges to 
$\V{V}(L)$ from the continuity of the volume.
Thus we have 
\begin{align*}
\V{V}
\big( L \big)& \geq 
\liminf _{\varepsilon \to 0}
\int_{U} 
\big(\theta |_{V} +  \varepsilon \omega +
\ddbar \Psi _{\varepsilon }  \big)_{\ac}^{d}\\
&\geq \int_{U} 
\liminf _{\varepsilon \to 0} 
\big(\theta |_{V} + \varepsilon \omega +
 \ddbar \Psi  _{\varepsilon } \big)_{\ac}^{d}\\
&= \int_{U} 
\big(\theta |_{V} +
\ddbar A|z|^{2} \big)_{\ac}^{d}.
\end{align*}
The second inequality follows form Fatou's lemma.
Here we used the fact that $U$ does not shrink 
even if $\varepsilon$ goes to $0$, 
since $U$ is independent of $\varepsilon$, 
 The equality follows from 
the property (\ref{4-2}).
By the choice of $A$ (see (\ref{0-0})), the right hand of the above inequality
$$ \int_{U} 
\big(\theta |_{V} +\ddbar A|z|^{2} \big)_{\ac}^{d}
= \int_{U} 
\big(\theta |_{V} +\ddbar A|z|^{2} \big)^{d}$$ is positive.
Hence we proved the volume $\V{V}(L)$ is positive 
for a non-singular subvariety $V$.

When $V$ has singularities, we take an embedded resolution
$\mu :\widetilde {V} \subseteq  \widetilde {X}  \longrightarrow V \subseteq X$. 
 Then we can show that 
${\rm{vol}}_{\it{\widetilde {V}}} (\mu^{*}L) >0$ 
by the same argument as above. 
Note that we used only the following property 
on the line bundle $L$ in the above argument:

$(*)$ For every point 
$y \in \partial B$, 
there exists a 
$\theta |_{V}$-psh function  
$\widetilde{\varphi} _{y}$  
such that $\widetilde{\varphi} _{y} (p) = -\infty$ 
and
$\widetilde{\varphi} _{y}(y)=0$.

We can easily show that 
the property $(*)$ holds for the pull-back $\mu^{*}L$ of $L$ 
as follows: 
We first choose a point $p$ on $\widetilde {V}$ such that 
$\mu$ is an isomorphism on a neighborhood $B$ of $p$.
For every point $y \in \partial B$, 
we consider a curve $C_{y}$ on $\widetilde {V}$ 
which contains the points $p$ and $y$.  
Since $\mu$ is an isomorphism on $B$, 
the push-forward $\mu(C_{y})$ is not a point. 
Therefore it follows that the intersection number 
$\big( L \cdot \mu(C_{y}) \big) $ is positive 
from Proposition \ref{Nef}.
Lemma \ref{potential} implies that there exists 
a $\theta |_{\mu(C_{y})}$-psh function $\varphi _{y}$ 
such that $\varphi _{y} \big( \mu (p) \big) = -\infty $ and
$\varphi _{y} (\mu(y))=0$.
By the assumption of Theorem \ref{Main} on $L$, 
we can extend $\varphi _{y}$ 
to a global $\theta$-psh function $\widetilde{\varphi} _{y}$ 
on $X$. 
Then the pull-back 
$\mu^{*} \widetilde{\varphi} _{y}$ of 
$\widetilde{\varphi} _{y}$
satisfies 
property $(*)$ which we want to obtain. 
By the same argument as above, we obtain 
${\rm{vol}}_{\it{\widetilde {V}}} (\mu^{*}L)  >0 $.
Since the restriction $\mu |_{\widetilde {V}}$ a birational morphism from $\widetilde {V}$ to $V$, 
$\V{V}(L) = {\rm{vol}}_{\it{\widetilde {V}}} (\mu^{*}L)$ 
(see \cite[Proposition 2.2.43]{La04}).
Hence we proved the volume $\V{V}(L) > 0$ 
for any subvariety $V$, 
even if $V$ has singularities.

Since $L$ is a nef line bundle, the volume of $L$ on $V$ 
coincides with the intersection number $(L^{d}\cdot V)$. 
Therefore $L$ is an ample line bundle by 
the Nakai-Moishezon-Kleiman criterion.

\section{Example}
\begin{ex}\label{Ex}
(This example shows us that there exists
a $\theta $-psh function on some subvariety
which can not extended to a global $\theta $-psh function
even if $L$ is semi-ample and big.)

Let $\pi: X:={\rm{Bl_{p}}}(\mathbb{P}^{2})\longrightarrow \mathbb{P}^{2}$ be a blow-up along a point 
$p \in \mathbb{P}^{2}$ 
and $L$ the pull-back of the hyperplane bundle by $\pi$.
Then $L$ is a semi-ample and big.
(However $L$ is not ample.)
We denote by $\theta $ the pull-back of the Fubini-Study form
on $\mathbb{P}^{2}$.
Note that $\theta $ represents the first Chern class 
$c_{1}(L)$ of $L$.
By the definition of $\theta $, the restriction of $\theta $
to $E$ is zero $(1,1)$-form, 
where $E$ is the exceptional divisor.
Therefore any $\theta |_{E}$-psh on $E$ is constant 
by the maximum principle of psh functions.
It says that a global $\theta $-psh function has the same value 
along $E$.

Now we denote by $C$ an irreducible curve on $X$ which 
intersects $E$ with at least two points.
Then $C$ is not contractive by $\pi$.
Therefore the degree of $L$ on $C$ is positive
by the projection formula
(that is, $L|_{C}$ is ample on $C$).
It implies that there exist 
many $\theta |_{C}$-psh functions on $C$.
In particular, there exists a $\theta |_{C}$-psh function which 
has different values at intersection points with $E$. 
Indeed, we can such function by using Lemma \ref{potential}. 
Such function can not extend to a global $\theta $-psh 
function.
\end{ex}

\begin{ex}\label{Ex2}
(Relations between the restricted volume of a line bundle and 
the extendability of $\theta $-psh functions.)\\ \ 
The restricted volume of $L$ along a subvariety 
$V$ is defined to be 
\begin{equation*}
\mathrm{vol_{ \it{X} | \it{V}}}(L)
 = \limsup_{k \to \infty} \dfrac
{\dim  H^{0} \big( X|V, \mathcal{O}(kL) \big) }{k^d/d!}, 
\end{equation*}
where 
\begin{equation*}
 H^{0} \big( X|V, \mathcal{O}(kL) \big) =
\mathrm{Im} \Big( H^{0}\big(X,\OX(kL) \big) \longrightarrow H^{0}\big(V,\OV(kL) \big) \Big). 
\end{equation*}

The restricted volume measures the number of sections of $\OV(kL)$ 
which can be extended to $X$.
Due to \cite{Mat10}, restricted volumes can be expressed 
with current integrations as follows 
(see \cite{Mat10} for details). 
\begin{theo}[{\cite[Theorem 1.2]{Mat10}}]
Assume that $V$ is not contained in the augmented base locus $ {\mathbb{B}}_{+}(L)$.
Then the restricted volume of $L$ along $V$ satisfies the following equality :
\begin{equation*}
\mathrm{vol_{\it{X} | \it{V} }}(L) = \sup _{\varphi \in \PshX} 
\int_{V_{\reg}} { (\theta |_{V_{\reg}}+ 
\ddbar \varphi |_{V_{\reg} })^{d}_{\ac}} ,
\end{equation*}
for $\varphi $ ranging among $\theta $-psh functions on $X$ with analytic singularities whose singular locus does not contain 
$V$.
\end{theo}
 The right hand integral measures 
Monge-Amp\`ere products of $\theta |_{V}$-psh functions on $V$
which can be extended to $X$.
On the other hand, the volume of the line bundle $L|_{V}$ 
can also be expressed 
with current integrations 
(see Proposition \ref{Bou} and \cite{Bou02}).
If any $\theta |_{V}$-psh function can be extended to a 
global $\theta $-psh function, 
the restricted volume along $V$ 
and the volume on $V$ coincides.
However there exists an example such that they are different
even if $V$ is not contained in the augmented base locus.

For example, when $X$ is a surface, 
a big line bundle $L$ admits a Zariski decomposition. 
That is, there exist nef $\mathbb{Q}$-divisor $P$ and 
effective $\mathbb{Q}$-divisor $N$ 
such that 
\begin{equation*}
H^{0}(X, \OX(\lfloor kP \rfloor)) \longrightarrow  H^{0}(X, \OX(kL))
\end{equation*}
is an isomorphism.
The map is multipling the section $e_{k}$, 
where $e_{k}$ is the standard section of the effective divisor 
$\lceil kN \rceil$. Here $\lfloor G \rfloor$ 
(resp. $\lceil G \rceil$) denotes round down (resp. round up) of an $\mathbb{R}$-divisor G.
Let $V$ be an irreducible curve which is not contained in 
the augmented base locus $ {\mathbb{B}}_{+}(L)$. 
Then $L|_{V}$ is an ample line bundle. 
Further, the restricted volume along $V$ is 
computed by the self-intersection number of 
the positive part $P$ along $V$
when $L$ admits a Zariski decomposition 
(see \cite[Proposition 3.1]{Mat10}).
That is, 
\begin{align*}
& \mathrm{vol_{\it{X} | \it{V} }}(L) = (P \cdot V),  \\ 
& \mathrm{vol_{\it{V}}}(L) = 
(L \cdot V ) = (P \cdot V) + (N \cdot V).
\end{align*}
Therefore, the volume and restricted volume 
may be different value
unless 
$(N \cdot V)$ is not equal to zero.
When $L$ is not nef 
(that is, $N$ is non-zero divisor) 
and $V$ is an ample divisor, 
$(N \cdot V)$ is not zero.
\end{ex}


\end{document}